\documentclass[12pt,letterpaper,english]{article}
\usepackage{mathptmx}
\usepackage{helvet}

\usepackage[T1]{fontenc}
\usepackage[latin9]{inputenc}
\usepackage{amsmath}
\usepackage{amssymb}
\usepackage{esint}
\usepackage{nameref}
\usepackage{titlesec}

\makeatletter

[1994/12/01]

\usepackage{amsthm}
\usepackage{parskip}
\usepackage{makeidx}
\usepackage{babel}

\titlelabel{\thetitle.\quad}

\newtheorem{thm}{Theorem}[section]

\newtheorem{prop}[thm]{Proposition}

\theoremstyle{definition}

\theoremstyle{remark}

\numberwithin{equation}{section}

\newcommand{\eval}[2][\right]{\relax
  \ifx#1\right\relax \left.\fi#2#1\rvert}

\numberwithin{equation}{section}

\title{Spectral Analysis of a Class of Self-Adjoint Difference Equations}
\author{Dale T. Smith, Ph.D.\\
Fiserv.\\
Risk and Compliance Division\\
107 Technology Parkway\\
Norcross, GA 30092}
\date{}

\begin{document}
\maketitle
\abstract{In this paper, we consider self-adjoint difference equations of the form
\begin{equation}
-\Delta\left(a_{n-1}\Delta y_{n-1}\right)+b_{n}y_{n}=\lambda y_{n},n=0,1,\dots\label{eq:abstract}
\end{equation}
where $a_{n-1}>0$ for all $n\ge0$ and $b_{n}$ are real and $\lambda$ is complex. Under the assumption that $a_{n-1}$ satisfies certain
growth conditions and is limit point (that is, the associated Hamburger moment problem is determined), we prove that the existence of an exponentially
bounded solution of \eqref{eq:abstract} implies a bound on the distance from $\lambda$ to the spectrum of the associated self-adjoint operator,
and that if a solution of \eqref{eq:abstract} is bounded by a power of n for n sufficiently large, then $\lambda\in\sigma(B)$. Here, $B$ is a certain self-adjoint operator generated by \eqref{eq:abstract}. These results are the difference equation version of differential operator results of Shnol'. We use this to then prove that the spectrum of the associated orthogonal polynomials contains the closure of the set of $\lambda$ for which we can find a polynomially bounded solution. We also present a result concerning the invariance of the essential spectrum under weak perturbation of the coefficients.} \\


\section{Introduction}
In this paper we consider the second order difference equation in self-adjoint form
\begin{equation}
-\Delta(a_{n-1}\Delta y_{n-1})+b_{n}y_{n}=\lambda y_{n},n=0,1,\dots,\label{eq1}
\end{equation}
 where $a_{n-1}>0$, and $b_{n}$ is real for n = 0, 1, \dots, and $\lambda$ is a possibly complex number. Here and throughout this paper, $\Delta$ is the forward difference operator defined by 
\begin{align*}
\Delta g_{n}=g_{n+1}-g_{n}.
\end{align*}
 We shall also assume that for any solution $y_{n}$ of \eqref{eq1}, $y_{-1}=0$. Equation \eqref{eq1} may be written in the form
\begin{equation}
-a_{n}y_{n+1}-a_{n-1}y_{n-1}+(b_{n}+a_{n}+a_{n-1})y_{n}=\lambda y_{n} \nonumber
\end{equation}
and by making the change of dependent variable $y_{n}\mapsto(-1)^{n}y_{n}$, we may write 
\begin{equation}\label{eq2}
a_{n}y_{n+1}+a_{n-1}y_{n-1}+(b_{n}+a_{n}+a_{n-1})y_{n}=\lambda y_{n}
\end{equation}
Using a theorem commonly attributed to Favard (see Ismail \cite{ismail-book}, p. 31), we see that there is a sequence of orthonormal polynomials $\{p(n;\lambda\}$
which satisfy \eqref{eq2} and there is a measure $d\mu(\lambda)$ such that 
\begin{align*}
\int_{-\infty}^{\infty}p(n;\lambda)p(m;\lambda)d\mu(\lambda)=\delta_{nm}
\end{align*}
Thus a solution of \eqref{eq1} is then $(-1)^{n}p(n;\lambda)$. Note that the polynomials $\{p(n;\lambda)\}$ satisfy the initial value problem consisting of \eqref{eq2} and the initial conditions $p(-1;\lambda)=0$ and $p(0;\lambda)=1$. There have been a number of papers in the past few years devoted to the study of orthogonal polynomials, particularly the interaction between the recurrence coefficients in \eqref{eq2}, the orthogonality measure $d\mu(\lambda)$, and the polynomials $\{p(n;\lambda)\}$.  For recent references, see \cite{ismail-book}.\\

\noindent There is an intimate connection between the spectral theory for \eqref{eq1}, orthogonal polynomials, continued fractions, asymptotic expansions of Stieltjes integrals, moment problems, variational inequalities, oscillation and disconjugacy, and a host of other topics The three volumes by Henrici \cite{henrici} are perhaps the best self-contained reference for the connections between orthogonal polynomials, continued fractions, Pade approximation, and some of the other topics mentioned above. Disconjugacy and oscillation for \eqref{eq1} is covered in Kelly and Peterson \cite{kelly-peterson}. In an earlier paper \cite{smith-spectral}, we discussed the connection between the oscillation of solutions of \eqref{eq1} and the essential spectrum of a linear operator generated by \eqref{eq1}; for the definition of this linear operator, see below. For other results on the oscillation of solutions of \eqref{eq1} and orthogonal polynomials, see van Doorn \cite{vandoorn-1} and \cite{vandoorn-2}. Mingarelli \cite{mingarelli} discusses second order difference equations from the point of view Volterra~Stieltjes integral equations, and unifies many of the results on oscillation, spectrum, and comparison of solutions proved separately for \eqref{eq1} and second-order ordinary differential equations in formal self-adjoint form. Other general references involving \eqref{eq1} are Ismail \cite{ismail-book} on orthogonal polynomials and related topics, and Jones and Thron \cite{jones-thron} on continued fractions and their relationship with Pade approximation, moment problems, and orthogonal polynomials. We have certainly not included all references to books or papers which contain something on these topics; that would be a rather large task. \\

\noindent In this paper, we pursue the connection between the spectral theory of \eqref{eq1} and exponential and polynomial upper
bounds on solutions of (1.l). It is important to note that we do not require estimates of the orthogonal polynomials, but that some solution of (1.1) have an appropriate estimate. We prove the difference equation versions of some theorems of Shnol' and Simon; see Glazman \cite{glazman}, p. 175-182 for proofs of Shnol's results in English in the elliptic differential equation case and \cite{simon-semigroups} for Simon's results. A special case of these results were announced in section 2 of \cite{smith-spectral} and \cite{smith-dissertation}, but the proofs of the results stated here are given here for the first time.
The results given in \cite{smith-spectral} and \cite{smith-dissertation} are very restrictive, since the author had to assume that (1.1) is nonoscillatory, $a_{n-1}=1$ for all n, and $b_{n}$ is bounded below. In this paper, we remove all the restrictions we just mentioned, but place growth constraints on the $a_{n}$. \\

\noindent We shall need the relevant operator theory, which is found in Hinton and Lewis \cite{hinton-lewis}. Let $H=\ell^{2}(Z^{+})$
be the Hilbert space of square-summable sequences with inner product
\[
(u,v)=\sum\limits _{n=0}^{\infty}u_{n}v_{n}^{*}.
\]
The associated norm will be denoted by $\|.\|$. Let the linear operator
$B:H\mapsto H$ be defined by 
\[
D(B)=\{y\in H:By\in H\}
\]
where 
\[
(By)_{n}=-\Delta(a_{n-1}\Delta y_{n-1})+b_{n}y_{n}
\]
The resolvent set of B is the set 
\[
\rho(B)=\{\lambda\in\mathbb{C}:B-\lambda\in\mathbb{B}(H)\},
\]
 ($\mathbb{B}(H)$ is the set of bounded linear operators on H) and the spectrum of B is the set $\sigma(B)=\mathbb{C}\setminus\rho(B)$. An eigenvalue of B is a (necessarily) real number $\lambda$ such that the equation $By=\lambda y$ has a non-zero solution in H. An eigenvalue is called a discrete eigenvalue if the associated eigenspace has finite dimension (that is, the eigenvalue has finite geometric multiplicity), and the set of discrete eigenvalues of B is denoted by $\sigma_{d}(B)$. The essential spectrum of B is the set $\sigma_{ess}(B)=\sigma(B)\setminus\sigma_{d}(B)$. Using the spectral theorem, we may characterize these parts of the spectrum in an alternative way: the essential spectrum is the set of real numbers which are of infinite geometric multiplicity or are limit points of the spectrum; the discrete spectrum is the set of points of the spectrum which are isolated. For more information on these terms and the spectral theory of linear operators, see Weidmann \cite{weidmann}.\\

\noindent Finally, equation \eqref{eq1} is said to be limit-circle if every solution of \eqref{eq1} with $Im(\lambda)\ne0$ is in H; \eqref{eq1} is said to be limit-point if it is not limit circle. Note that this classification is independent of $\lambda$ (see \cite{akhiezer}). For a discussion of these terms and the reason
behind the names limit-circle/limit point, see Chapter 1 of Akhiezer \cite{akhiezer}. An important result for us is the following proposition, which is a combination of Theorem 2 in Hinton and Lewis \cite{hinton-lewis} and Corollary 2.2.4 and Theorem 2.1.2 in Akhiezer \cite{akhiezer}. 
\begin{prop}\label{prop-1} The linear operator B is self adjoint if and only if \eqref{eq1} is limit point if and only if the associated Hamburger moment problem is determined.
\end{prop}
\noindent In the limit-circle case (or if the associated moment problem is not determined), the situation is more complicated. In that case the operator B is not self adjoint by Proposition 1.1, but a restriction of B may have self adjoint extensions. Welstead (\cite{welstead-bc} and \cite{welstead-sa}) and Shi and Sun \cite{shi-sun} have shown that a boundary condition at infinity is needed to construct the self adjoint extensions of this restriction of B. For the abstract theory of self adjoint extensions of symmetric operators, see Weidmann \cite{weidmann}.

\section{The Main Results and Discussion}

The main results are as follows.

\begin{thm}\label{thm-1} Suppose there are constants $C_{1},\gamma,n_{0},L>0$
such that for all $n>n_{0}$ 
\begin{align*}
\left|\Delta a_{n-1}\right|\le C_{1}a_{n-1},\sum\limits _{k=n_{0}+1}^{n}a_{k-1}^{2}\le Le^{\gamma n}
\end{align*}
 and suppose \eqref{eq1} is limit point. Suppose there are
numbers $\beta,C_{2}>0$ and a solution $y_{n}$ of \eqref{eq1}
such that for all $n>n_{0}$ 
\[
\left|y_{n}\right|\le C_{2}e^{\beta n}
\]
 and that $\lambda$ is real. There there is a constant $C=C(\lambda,C_{1},C_{2})$
such that 
\[
d(\lambda,\sigma(B))\le C\left(e^{2\beta}-1\right)^{1/2}
\]
 where $d(\lambda,\sigma(B))$ is the distance between $\lambda$
and the spectrum of B. \end{thm}

\begin{thm}\label{thm-2} Suppose that the first two hypotheses of
Theorem \ref{thm-1} are true, but suppose there is a solution $y_{n}$
of \eqref{eq1} such that for \emph{any} $\beta>0$ there is
a finite constant $C_{3}(\beta)>0$ with 
\[
\left|y_{n}\right|\le C_{3}(\beta)e^{\beta n}
\]
 for all $n>n_{0}$ and that $\lambda$ is real. Then $\lambda\in\sigma(B)$.
In particular, if there is a real number $\theta$ and a constant
$C_{4}(\lambda)$ such that 
\[
\left|y_{n}\right|\le C_{4}n^{\theta}
\]
 for all $n>n_{0}$, then $\lambda\in\sigma(B)$. \end{thm}

\begin{thm}\label{thm-3} Suppose there are constants $C_{1},\gamma,n_{0},L>0$
such that for all $n>n_{0}$ 
\begin{align*}
\left|\Delta a_{n-1}\right|\le C_{1}a_{n-1},\sum\limits _{k=n_{0}+1}^{n}a_{k-1}^{2}\le Le^{\gamma n}
\end{align*}
 and suppose \eqref{eq1} is limit-point. Let
\[
E=\{\lambda\in\mathbb{R}:\thickspace has\thickspace a\thickspace polynomially\thickspace bounded\thickspace solution\}.
\]
 Then $\bar{E} \subset \sigma(B)$. \end{thm}
 
\noindent We must again point out that these theorems do not require an estimate on the orthogonal polynomials obtained from \eqref{eq1}. Also, these results should be compared with those of Smith \cite{smith-spectral} and \cite{smith-dissertation}. Our results here do not require a hypothesis on the $b_{n}$ (other than that it is a real sequence) and the hypotheses on the $a_n$ are weak enough to allow for many interesting examples, such as polynomial ore exponential behavior.\\

\noindent Theorems \ref{thm-1}, \ref{thm-2} and \ref{thm-3} should be compared to the differential equation results found in \cite{glazman} (starting on p. 176) and \cite{simon-semigroups}. In \cite{glazman}, the potential $q$ of an elliptic differential operator $-A+q$ is required to be bounded below, but this condition is weakened in \cite{simon-semigroups}. A key part of the proof of Theorem \ref{thm-2} is Theorem 2.1 (see below), which gives an a-priori bound on $\Delta y_{n-1}$ in terms of $y_n$. There are corresponding theorems for the differential equation case in \cite{glazman}, p. 176-178 and \cite{simon-semigroups}. The proof of Theorem \ref{thm-3} relies on a generalization of Theorem 9, p. 197 of \cite{glazman} and is related to Stieltjes conjecture (see section 58 of \cite{glazman}). Several years ago, Simon \cite{simon-semigroups} proved a theorem like \ref{thm-3} for a class operators related to the N-body problem of quantum mechanics.\\

\noindent The reader may suspect that the assumption of self-adjointness in Theorems \ref{thm-1}, \ref{thm-2} and \ref{thm-3} is an artifact of our method, and may be removed by using other methods. But self-adjointness seems to be essential to prove theorems such as these, even in the differential equations case. The reason for this is as follows. In the limit-circle case it is known that self-adjoint extensions of a restriction $B_{0}$ of the operator B may be considered as finite-dimensional perturbations of $B_{0}$ and thus (by Weyl's theorem) have the same essential spectrum (see \cite{glazman}, p. 10, 44-46). But discrete eigenvalues of an abstract operator may in fact disappear if the operator is perturbed by an arbitrarily small bounded operator (see Kato \cite{kato}; the theorem is known as the Weyl-Von Neumann Theorem). Indeed, one of the earliest quantum mechanical models exhibits the phenomenon of loss of discrete eigenvalues - the Stark effect (a hydrogen atom in a uniform electric field; here the perturbation is unbounded, see Simon). Thus, if discrete eigenvalues are present, the spectra of the various self adjoint extensions of a restriction of B may be quite different. Because of this it is not clear whether the results of this paper hold in the limit-circle case. A condition we do not investigate is the essential self-adjointness of B. These results would probably require us to know some form of relative boundedness for \eqref{eq1}. We do not know if such results yet exist in the literature. For a discussion of relative boundedness and its connection with essential self-adjointness, see Weidmann \cite{weidmann}

\section{Proof of the Main Results}

The following theorem plays a key role in the proof of Theorem \ref{thm-1}, as does the corresponding result for the differential equations case; see Simon \cite{simon-semigroups}.

\begin{thm}\label{thm-4} Suppose that there is a constant $C_{1}>0$
such that for all n sufficiently large 
\begin{align}\label{eq3}
\left|\Delta a_{n-1}\right|\le C_{1}a_{n-1}
\end{align}
 and suppose $y_{n}$ is any sequence (not necessarily a solution of \eqref{eq1}). Then there is an integer m, and for integers r, s, and n which satisfy $m\le r\le s\le n$, we have 
\begin{align}\label{eq4}
\sum\limits _{k=r}^{s}a_{k-1}^{2}\left(\Delta y_{k-1}\right)^{2}\le2\left(1+\left(C_{1}+1\right)^{2}\right)\sum\limits _{k=m-1}^{n}a_{k-1}^{2}y_{k}^{2}.
\end{align}
 \end{thm}

\begin{proof}First, choose m large enough such that 
\begin{align}\label{eq5}
\left|\Delta a_{k-1}\right|\le C_{1}a_{k-1} \text{ for } k \ge m-1.
\end{align}
 For $k\ge m-1$, \eqref{eq5} implies that 
\begin{align}
\left|\frac{a_{k}}{a_{k-1}}-1\right|\le C_{1}, \nonumber
\end{align}
 or 
\begin{equation}\label{eq6}
1-C_{1}\le\frac{a_{k}}{a_{k-1}}\le1+C_{1}.
\end{equation}
 Now 
\begin{equation}
\sum\limits _{k=m}^{n}a_{k-1}^{2}\left(\Delta y_{k-1}\right)^{2}\le\sum\limits _{k=r}^{s}a_{k-1}^{2}\left(\left|y_{k-1}\right|+\left|y_{k}\right|\right)^{2}. \nonumber
\end{equation}
By first squaring on the right of this last inequality, and then using the inequality $2ab \le a^{2}+b^{2}$ on the cross-terms, we have 
\begin{align}\label{eq7}
\sum\limits _{k=m}^{n}a_{k-1}^{2}\left(\Delta y_{k-1}\right)^{2}\le2\sum\limits _{k=m-1}^{n-1}a_{k}^{2}y_{k}^{2}+2\sum\limits _{k=m}^{n}a_{k-1}^{2}y_{k}^{2}.
\end{align}
In the first some on the right of \eqref{eq7} use 
\begin{equation}
a_{k}=a_{k-1}\frac{a_{k}}{a_{k-1}} \nonumber
\end{equation}
 and \eqref{eq6} to get 
\begin{equation}
\sum\limits _{k=m}^{n}a_{k-1}^{2}\left(\Delta y_{k-1}\right)^{2}\le2\left(1+\left(C_{1}+1\right)^{2}\right)\sum\limits _{k=m-1}^{n}a_{k}^{2}y_{k}^{2}.
\end{equation}
\end{proof}

\noindent We would like to note here that the proof of \ref{thm-4} is much shorter and easier than the proof given in \cite{smith-dissertation} and the proof for the differential equation case on p. 176-178 in \cite{glazman}. Also, we do not require $y_{n}$ to be a solution, as does \cite{glazman}, \cite{smith-dissertation}, and \cite{smith-spectral}.

\subsection{Proof of Theorem \ref{thm-1}}

The proof of Theorem \ref{thm-1} relies on using $[U,V]z=U(Vz)-V(Uz)$, the commutator of U and V, which is defined for $z\in D(U)\cap D(V)$. In the theory of the Schrodinger equation, commutator estimates such as those used below have been found to be very useful; see Simon \cite{simon-semigroups}. The results are in Section C.4 of \cite{simon-semigroups}. Commutator estimates have also been used in the theory of Jacobi matrices; see \cite{dombrowski-commutator}. The proof of Shnol's theorem found in Simon \cite{simon-semigroups} is where we learned the idea of using the commutator. The complication we have with the difference equation case in this paper is mainly due to the product rule, which is
\[
\Delta\left(s_{n}t_{n}\right)=s_{n+1}\Delta t_{n}+t_{n}\Delta s_{n}
\]
or
\[
\Delta\left(s_{n}t_{n}\right)=s_{n}\Delta t_{n}+t_{n+1}\Delta s_{n}
\]
We also define the shift operator $(Eu)_{n}=u_{n-1}$ and the norm on the integers in the interval $[s,t]$ by
\[
\|u\|_{[s,t]}^{2}=\sum\limits_{n=s}^{t}u_{n}^{2}
\]
\begin{proof}[Proof of Theorem \ref{thm-1}]
Let $n_{0}$ be a positive integer chosen such that for all $n\ge n_{0}+1$
\begin{equation}
\left|\Delta a_{n-1}\right|\le C_{1}a_{n-1} \text{ and } \left|y_{n}\right|\le C_{2}e^{\beta n}. \nonumber
\end{equation}
 Let $r-1>n_{0}$ and let $v_{r}$ be the sequence 
\begin{align*}
v_{r}=v_{r,n}= & \begin{cases}
1 & \text{ if } n_{0}\le n\le r,\\
0 & \text{ otherwise.}
\end{cases}
\end{align*}
In what follows, $\Delta$ is differencing with respect to $n$. Now $v_{r}y$ is in the domain of $B$ (because it has finite support), so
\[
\left[B,v_{r}\right]\left( v_r y \right)=B\left(v_{r}y\right)-v_{r}By
\]
or
\[
B\left(v_{r}y\right)=\left[B,v_{r}\right]y+\lambda v_{r}y.
\]
by using the fact that $y$ satisfies $(B-\lambda)y=0$. Thus,
\begin{equation}
(B-\lambda)(v_{r}y)=\left[B,v_{r}\right]y \nonumber \label{thm-1-eq0}
\end{equation}
Using the first product rule above, we have
\[
\left[B,v_{r}\right]y_{n}=-\Delta\left(a_{n-1}\Delta\left(v_{r,n-1}y_{n}\right)\right)+b_{n}v_{r,n}y_{n}+\left(b_{n}-\lambda\right)v_{r,n}y_{n}.
\]
Expanding this out by using the product rule again gives us
\begin{eqnarray}
\left[B,v_{r}\right]y_{n} = -\Big( v_{r,n} \Delta \left( a_{n-1} \Delta y_{n-1} \right) & + & a_n \Delta y_n \Delta v_{r,n} + \Delta v_{r,n} \Delta \left( a_{n-1} y_{n-1} \right) \nonumber \\
& + & a_{n-1}y_{n-1} \Delta^2 v_{r,n-1} \Big) \nonumber \\
& + & \left( b_n - \lambda \right) v_{r,n} y_n.
\end{eqnarray}
But in this last equation, the first and last terms on the right add together to give zero, since $y$ is a solution of \eqref{eq1}. Thus
\[
\left[B,v_{r}\right]y_{n}=-\Delta v_{r,n}\left(a_{n}\Delta y_{n}+\Delta\left(a_{n-1}y_{n-1}\right)\right)-a_{n-1}y_{n-1}\Delta^{2}v_{r,n}
\]
Using the product rule again for $\Delta$ gives
\begin{align*}
\left[B,v_{r}\right]y_{n}=-\Delta v_{r,n}\left(a_{n}\Delta y_{n}+y_{n}\Delta a_{n - 1}+a_{n-1}\Delta y_{n-1}\right)-a_{n-1}^{2}y_{n-1}^{2}\Delta^{2}v_{r,n-1}.
\end{align*}
We need $\left(\left[B,v_{r}\right]y_{n}\right)^{2}$. This is an easy, but tedious operation. We first square both sides, and in the cross terms, use the inequality $2ab\le a^{2}+b^{2}$ to get
\begin{multline}
\left(\left[B,v_{r}\right]y_{n}\right)^{2} \le C\,\left(\Delta v_{r,n}\right)^{2}\left\{a_{n}^{2}\left[\left(\Delta y_{n}\right)^{2}+\left(\Delta y_{n-1}\right)^{2}\right]+y_{n-1}^{2}\left(\Delta a_{n-1}\right)^{2}\right\}\\
+ C\, a_{n-1}^{2}y_{n-1}^{2}\left(\Delta^{2}v_{r,n-1}\right)^{2}.\label{thm-1-eq1}
\end{multline}
Now we may write this using the norm as
\begin{multline}\label{thm-1-eq2} \left(\left[B,v_{r}\right]y_{n}\right)^{2} \le C \,\Big\{\|a\Delta y\|_{[r,r+3]}^{2}+\|a\left(\Delta Ey\right)\|_{[r,r+3]}^{2}\\
+\|\Delta(Ea)\, Ey\|_{[r,r+3]}^{2}+\|Ea\, Ey\|_{[r-1,r+3]}^{2}\Big\} \end{multline} 
Note that we have used the following facts
\begin{eqnarray*}
supp\left(\Delta\left(Ev_{r}\right)\right) & \subset & integers\; in\; the\; interval\,\left[r-1,r+2\right];\\
supp\left(\Delta^{2}\left(Ev_{r}\right)\right) & \subset & integers\; in\; the\; interval\,\left[r-1,r+3\right];
\end{eqnarray*}
in obtaining \eqref{thm-1-eq1} from \eqref{thm-1-eq2} and the constant $C > 0$ in these last inequalities does not depend on $n$. These may be derived from the definition of the sequence $v_{r}$. Now, by changing the index of summation,
\begin{equation}
\|a\Delta y\|_{[r,r+3]}^{2}=\|Ea\,\Delta Ey\|_{[r+1,r+4]}^{2}\le\|Ea\,\Delta Ey\|_{[r,r+4]}^{2}.\label{thm-1-eq3}
\end{equation}
Also, by the hypothesis on $a_{n-1}$,
\begin{equation}
\|a\Delta Ey\|_{[r,r+3]}^{2}\le\left(C_{1}+1\right)^{2}\|Ea\Delta Ey\|_{[r,r+3]}^{2}\label{thm-1-eq4}
\end{equation}
Using first \eqref{thm-1-eq3} and \eqref{thm-1-eq4} in \eqref{thm-1-eq2}, and then using Theorem \nameref{thm-1}twice, we have
\begin{equation}
\|\left[B,v_{r}\right]y\|^{2}\le K\left\{\|Ea\, y\|_{[r-1,r+4]}^{2}+2\|Ea\, y\|_{[r,r+3]}^{2}+\|Ea\, y\|_{[r-1,r+3]}^{2}\right\}, \label{thm-1-eq5}
\end{equation}
where $K$is a positive constant. But \eqref{thm-1-eq1} allows us to write this last equation immediately above as
\begin{equation}
\|\left(B - \lambda\right)v_{r}y\|^{2}\le K\left\{\|Ea\, y\|_{[r-1,r+4]}^{2}+2\|Ea\, y\|_{[r,r+3]}^{2}+\|Ea\, y\|_{[r-1,r+3]}^{2}\right\}. \label{thm-1-eq6}
\end{equation}
We now define, for $r\ge n_{0}+1$,
\begin{equation}
F\left(r\right)=\sum\limits_{n=n_{0}+1}^{r}a_{n-1}^{2}y_{n}^{2}.
\end{equation}
Using $\eqref{thm-1-eq0}$, $\eqref{thm-1-eq5}$ may be written as
\begin{equation}
\|\left(B-\lambda\right)\left(v_{r}y\right)\|^{2}\le K\left(4F\left(r+4\right)-3F\left(r-2\right)\right)
\end{equation}
or
\begin{equation}
\|\left(B-\lambda\right)\left(v_{r}y\right)\|^{2}\le K_{1}\left(F\left(r+4\right)-F\left(r-2\right)\right)\label{thm-1-eq6}
\end{equation}
Using the exponential bound on $y_{n}$ gives us that for any $\delta>0$ there is an $r_{0}>n_{0}+1$ such that for $r \ge r_{0}$
\begin{equation}
F\left(r\right) \le e^{\left( 2\beta + \delta \right) r} \sum \limits_{n = n_0 + 1}^{r} a_{n-1}^2.
\end{equation}
We now show that for a given $\delta_1 > 0$ there is a sequence of integers $r_p \to \infty$ such that
\begin{equation}\label{thm-1-eq7}
F\left(r_p + 4 \right) < e^{2\beta + \delta_1}F\left(r_p - 2 \right)
\end{equation}
(strict inequality is important here as we shall see below). If no such sequence exists, then there is an $r_1 \ge n_0 + 1$ such that for all $r \ge r_1$
\begin{equation}
F\left(r + 4 \right) \ge e^{2\beta + \delta_1}F\left(r - 2\right) \label{thm-1-eq-8}
\end{equation}
From this it follows that there is a constant $M > 0$ such that for all $r$ large enough
\begin{equation}
F\left(r \right) \ge M e^{(2 \beta + \delta_1)r}. \label{thm-1-eq-9}
\end{equation}
Using $\eqref{thm-1-eq7}$ and $\eqref{thm-1-eq8}$, and the hypothesis on $\sum a_{n-1}^2$, we have
\begin{equation}
e^{(\delta_1 - \delta - \gamma)r} \le \frac{1}{LM} \mbox{ for } r \ge \max{r_0, r_1}.
\end{equation}
We now choose $\delta_1 > \delta + \gamma$, which gives us a contradiction. Thus, our assertions involving \eqref{thm-1-eq7} is true. Now translating this back to the norm and operator gives us that
\begin{equation}\label{thm-1-eq8}
\|\left( B - \lambda \right) w_p \|^2 < K_1 \left( e^{2 \beta + \delta_1} - 1 \right),
\end{equation}
where $w_p = w_{p,n} = v_{r_p,n} y_n$. Now $w_p$ converges weakly to zero but $\|w_p\| = 1$, so $w_p$ is not compact but is bounded. Using Theorem 10, p. 14 of Glazman \cite{glazman} on \eqref{thm-1-eq8} proves the theorem since $\delta > 0$ is arbitrary.
\end{proof}

\begin{proof}[Proof of Theorem \ref{thm-2}]
The first statement follows directly from Theorem \ref{thm-1}. To prove the second statement, we need only show there is an exponential bound, and apply the first statement of the theorem. Now for any $\beta > 0$, if $n > n_0$, then
\begin{equation}
\left | y_n \right | \le C_3 n^{\theta} = C_3 \frac{n^{\theta}}{e^{\beta n}} e^{\beta n} \nonumber
\end{equation}
But this last term is bounded independently in n.
\end{proof}
\noindent The following theorem is proved in Glazman \cite{glazman}, p. 197, Theorem 9.

\begin{thm}\label{thm5} For any $\epsilon > 0$ and for $\mu$-almost all values of $\lambda$ (where $\mu$ is the spectral measure of B), there is a constant $L(\epsilon, \lambda) > 0$ such that for all n
\begin{equation}
\left | (-1)^n p(n;\lambda) \right | \le L(\epsilon, \lambda) n^{1/2 + \epsilon}. \nonumber
\end{equation}
Here the $p(n;\lambda)$ are the orthogonal polynomials obtained from \eqref{eq2}.
\end{thm}

\noindent But we need more than this to prove our conjecture that under the hypotheses of Theorem \ref{thm-1}, $\bar{E} = \sigma(B)$. \\


\noindent To show these results are non-trivial, we consider a special case of an example found in Wimp \cite{wimp}. Consider the difference equation
\begin{equation}
-n(n + 1) y_{n+1} + 2 n^2 y_n - n(n - 1)y_{n - 1} = -n \lambda y_n. \label{orig-ex}
\end{equation}
In \cite{wimp} it is shown there are two linearly independent solutions of \eqref{orig-ex} such that
\begin{align}
y_{1,n} & = n^{-3/4} e^{2\sqrt{-n \lambda}} (1 + o(1)), \\
y_{2,n} & = n^{-3/4} e^{-2\sqrt{-n \lambda}} (1 + o(1)) 
\end{align}
We shall change the dependent variable using the idea of Hinton and Lewis \cite{hinton-lewis}. For a difference equation of the form
\begin{equation}
-p_n y_{n + 1} + q_n y_n - p_{n - 1}y_{n - 1} = \lambda c_n y_n \nonumber
\end{equation}
let
\begin{align*}
\tilde{y}_n & = \sqrt{c_n} y_n, \\ \nonumber
a_n & = c_n^{-1/2} p_n c_{n + 1}^{-1/2}, \\ \nonumber
b_n & = c_n^{-1} q_n. \nonumber
\end{align*}
Then equation in $\tilde{y}$ is of the form \eqref{eq2}. Note that this change of variables is related to the concept of equivalent continued fractions; see \cite{jones-thron}. Applying this to \eqref{orig-ex} gives the equation
\begin{equation}\label{mod-ex}
-(n(n + 1)^{1/2} \tilde{y}_{n + 1} + 2n \tilde{y}_n - (n(n - 1))^{1/2} \tilde{y}_{n - 1} = -\lambda \tilde{y}_n
\end{equation}
which has two linearly independent solutions
\begin{align*}
\tilde{y}_{1,n} & = n^{-1/4} e^{2\sqrt{-n \lambda}} (1 + o(1)), \\ \nonumber
\tilde{y}_{2,n} & = n^{-1/4} e^{-2\sqrt{-n \lambda}} (1 + o(1))  \nonumber
\end{align*}
The sequence $a_n = (n(n + 1))^{1/2}$ satisfies the hypotheses of Theorem \ref{thm-1} and \eqref{mod-ex} is limit-point (see \cite{akhiezer}). For all $\lambda$ there is a bounded solution. Our results show that the spectrum of the associated self-adjoint operator is $(-\infty, \infty)$. This result appears to be new.

\section{Invariance of the Essential Spectrum}
In this section, we consider the difference equation
\begin{equation}\label{perturb-diff}
-\Delta \left[ \left( a_{n-1} + \eta_{n-1} \right) \Delta y_{n-1} \right] + \left( b_n + \psi_n \right) y_n = \lambda y_n
\end{equation}
as a perturbation of \eqref{eq1}. Here, we require that $a_{n-1} + \eta_{n-1} > 0$ for all $n$. The result we present is applicable even if \eqref{eq1} or \eqref{perturb-diff} do not generate self-adjount operators or whether the associated moment problems are determined or not.\\

\noindent To state the theorem, we need to extend our operator-theoretic discussion found in the introduction using \cite{hinton-lewis}. Let $A : H \in H$ deonte the operator defined using \eqref{perturb-diff}, that is,
\begin{equation}
D(A) = \{ y \in H : Ay \in H\}
\end{equation}
where $Ay_n$ is the left-hand side of \eqref{perturb-diff}. We now define the minimal operators $A_0$ and $B_0$ as follows. Let $T = A$ or $B$ (the same discussion works for both) and let
\begin{equation}
D_1 = \{ y \in H : \mbox{ only finitely many of the } y_n \mbox{ are non-zero} \}. \nonumber
\end{equation}
Now let
\begin{equation}
T_1 = T \big |_{D_{0}} \nonumber
\end{equation}
be the restriction of $T$ to $D_1$. A calculation shows that $T_1$ is symmetric and densely defined, so it has a closure $T_0$. This minimal operator may or may not have self-adjoint extensions. For the relevant terms and theorems which we have used here, see Weidmann \cite{weidmann}.

\begin{thm}\label{thm5}
Suppose $b_n \ge \alpha > 0$ for all $n$ sufficiently large and
\begin{equation}
\lim_{n \to \infty} \frac{\eta_n}{a_n} = \lim_{n \to \infty} \frac{\psi_n}{b_n} = 0 \nonumber
\end{equation}
\end{thm}

\noindent This result is the difference equation analogue of the result for differential equations found in Glazman \cite{glazman}. The proof of Theorem \ref{thm5} is not given here, since it closely resembles the proof of Theorem 25 in section 8 of \cite{glazman}. The key to the proof is to show that $A_0 - B_0$ is relatively compact with respect to $B_0$ and then to use Weyl's theorem on the invariance of the essential spectrum under relatively compact perturbations; see Weidmann \cite{weidmann} and Chapter 1 of \cite{glazman}.\\

\newpage{}


\begin{thebibliography}{10}

\bibitem{akhiezer}Akhiezer, N. I., \emph{The Classical Moment Problem and some Related Questions in Analysis}, New York, Hafner Pub. Co., 1965.

\bibitem{dombrowski-commutator}Joanne Dombrowski, \emph{A commutator approach to absolute continuity for unbounded Jacobi operators}, J. Math. Anal. Appl., \textbf{378}, no.~1, (2011) p. 133--139.

\bibitem{glazman}Glazman, I. M., \emph{Direct Methods of Qualitative Spectral Analysis}, Jerusalem, Israel Program for Scientific Translations,
1965.

\bibitem{henrici}Henrick, P., \emph{Applied and Computational Complex Variables}, in three volumes, New York, Wiley, 1974.

\bibitem{hinton-lewis}Hinton, D. B. and Lewis, R. T., \emph{Spectral analysis of second order linear difference equations}, J. Math. Anal. Appl., \textbf{63} (1978), p. 421-438.

\bibitem{ismail-book}Ismail, M. E. H., \emph{Classical and Quantum Orthogonal Polynomials in One Variable}, 2005, Cambridge University Press.

\bibitem{janas-moszynski}Janas, J. and Moszyski, M., \emph{Spectral properties of Jacobi matrices by asymptotic analysis}, J. Approx. Theory, \textbf{120}, no.~2, (2003), p. 309--336.

\bibitem{janas-naboko-stolz}Janas, J, Naboko, S, and Stolz, G., \emph{Decay Bounds on Eigenfunctions and the Singular Spectrum of Unbounded Jacobi
Matrices}, Int Math Res Notices, \textbf{4} (2009), p. 736-764.

\bibitem{jones-thron}Jones, W. B. and Thron, W. J., \emph{Continued Fractions. Analytic Theory and Applications}, Cambridge University Press, 2009

\bibitem{kato}Kato, T., \textit{Perturbation Theory for Linear Operators}, Springer-Verlag, 1966.

\bibitem{kelly-peterson}Kelly, W. G. and Peterson, A. C., \emph{Difference Equations. An Introduction with Applications}, San Diego, Academic Press, 2000.

\bibitem{mingarelli}Mingarelli, A. B., \emph{Voltrerra-Stielties Integral Equations and Generalized Ordinary Differential Expressions}, Lecture Notes in Mathematics no. 989, Berlin, Springer-Verlag, 1983.

\bibitem{sahbani}Sahbani, J., \emph{Spectral theory of certain unbounded Jacobi matrices}, J. Math. Anal. Appl., \textbf{342}, no. 1, (2008), p. 663-681.

\bibitem{shi-sun}Shi, Y. and Sun, H., \emph{Self-adjoint extensions for second-order symmetric linear difference equations}, Lin. Alg. Appl., \textbf{434}, no. 4, (2011), p. 903-930.

\bibitem{simon-semigroups}Simon, B., \emph{Schrodinger semigroups}, Bull. Amer. Math. Soc., \textbf{7} (1982), p. 447-526.

\bibitem{smith-spectral}Smith, D. T., \emph{On the spectral analysis of self adjoint operators generated by second order difference equations}, Proc. Roy. Soc. Edinburgh, \textbf{118A} (1991), p. 139-151.

\bibitem{smith-matrix}Smith, D. T., \emph{Exponential decay of resolvents and discrete eigenfunctions of banded infinite matrices}, J. Approx. Theory, \textbf{68} (1991), p. 83-97.

\bibitem{smith-dissertation}Smith, D. T., \emph{Exponential Decay of Resolvents of Banded Infinite Matrices and Asyrnptotics of Linear Difference Equations}, Dissertation, Georgia Institute of Technology, March, 1990.

\bibitem{vandoorn-1}van Doorn, E. A., \emph{On oscillation properties and the interval of orthogonality of orthogonal polynomials}, SIAM J. Math., \textbf{15} (1984), p. 1031-1042.

\bibitem{vandoorn-2}van Doorn, E. A., \emph{A note on orthogonal polynomials and oscillation criteria for second order linear difference equations}, J. Math. Anal. Appl., \textbf{12} (1986), p. 354-359.

\bibitem{weidmann}Weidmann, J., \emph{Linear Operators in Hilbert Spaces}, Graduate Texts in Mathematics, no 68, New York, Springer-Verlag, 1980.

\bibitem{welstead-bc}Welstead, S. I., \emph{Boundary conditions at infinity for difference equations of limit circle type}, J. Math. Anal. Appl., \textbf{89} (1982), p. 442-461.

\bibitem{welstead-sa}Welstead, S. J., \emph{Self-adjoint extensions of Jacobi matrices of limit circle type}, J. Math. Anal. Appl., \textbf{89} (1982), p. 315-326.

\bibitem{wimp}Wimp, J., \emph{Computation with Recurrence Relations}, London, Pitman Publishing Ltd., 1984.\end{thebibliography}
\end{document}